\documentclass{tac} 
\usepackage[centertags,sumlimits,intlimits,namelimits,reqno]{amsmath}
\usepackage{latexsym,amsfonts,amssymb,exscale,enumerate}
\usepackage[colorlinks,linkcolor=darkblue,citecolor=darkblue,urlcolor=darkblue,breaklinks=true,pagebackref]{hyperref}
\usepackage{tikz,tikz-cd}
\usepackage{float,color,multirow,array,multicol}
\usepackage[all,cmtip]{xy}

\definecolor{darkblue}{rgb}{0,0,0.7} 

\newcommand\maps{{\colon}}
\newcommand{\define}[1]{{\bf \boldmath #1}}

\newcommand\rel{{\mathbf{Rel}}}
\newcommand\Span{{\mathbf{Span}}}
\newcommand\corel{{\mathbf{Corel}}}
\newcommand\cospan{{\mathbf{Cospan}}}
\newcommand\finset{{\mathbf{FinSet}}}
\newcommand\Set{{\mathbf{Set}}}
\newcommand\im{{\mathrm{im}}}

\usetikzlibrary{backgrounds,circuits,circuits.ee.IEC,shapes,fit,matrix}

\pgfdeclarelayer{edgelayer}
\pgfdeclarelayer{nodelayer}
\pgfsetlayers{edgelayer,nodelayer,main}

\tikzset{font=\footnotesize}
\tikzstyle{none}=[inner sep=0pt]
\tikzstyle{circ}=[circle,fill=black,draw,inner sep=3pt]

\newcommand{\mult}[1]
{
\begin{aligned}
    \resizebox{#1}{!}{
\begin{tikzpicture}
	\begin{pgfonlayer}{nodelayer}
		\node [style=none] (0) at (1, -0) {};
		\node [style=circ] (1) at (0.125, -0) {};
		\node [style=none] (2) at (-1, 0.5) {};
		\node [style=none] (3) at (-1, -0.5) {};
	\end{pgfonlayer}
	\begin{pgfonlayer}{edgelayer}
		\draw[line width=2pt] (0.center) to (1.center);
		\draw[line width=2pt] [in=0, out=120, looseness=1.20] (1.center) to (2.center);
		\draw[line width=2pt] [in=0, out=-120, looseness=1.20] (1.center) to (3.center);
	\end{pgfonlayer}
      \end{tikzpicture}}
\end{aligned}
}

\newcommand{\unit}[1]
{
  \begin{aligned}
    \resizebox{#1}{!}{
\begin{tikzpicture}
	\begin{pgfonlayer}{nodelayer}
		\node [style=none] (0) at (1, -0) {};
		\node [style=none] (1) at (-1, -0) {};
		\node [style=circ] (2) at (0, -0) {};
	\end{pgfonlayer}
	\begin{pgfonlayer}{edgelayer}
		\draw[line width=2pt] (0.center) to (2);
	\end{pgfonlayer}
      \end{tikzpicture}}
  \end{aligned}
}

\newcommand{\comult}[1]
{
\begin{aligned}
    \resizebox{#1}{!}{
\begin{tikzpicture}
	\begin{pgfonlayer}{nodelayer}
		\node [style=none] (0) at (-1, -0) {};
		\node [style=circ] (1) at (-0.125, -0) {};
		\node [style=none] (2) at (1, 0.5) {};
		\node [style=none] (3) at (1, -0.5) {};
	\end{pgfonlayer}
	\begin{pgfonlayer}{edgelayer}
		\draw[line width=2pt] (0.center) to (1.center);
		\draw[line width=2pt] [in=180, out=60, looseness=1.20] (1.center) to (2.center);
		\draw[line width=2pt] [in=180, out=-60, looseness=1.20] (1.center) to (3.center);
	\end{pgfonlayer}
      \end{tikzpicture}}
\end{aligned}
}

\newcommand{\counit}[1]
{
  \begin{aligned}
    \resizebox{#1}{!}{
\begin{tikzpicture}
	\begin{pgfonlayer}{nodelayer}
		\node [style=none] (0) at (-1, -0) {};
		\node [style=none] (1) at (1, -0) {};
		\node [style=circ] (2) at (0, -0) {};
	\end{pgfonlayer}
	\begin{pgfonlayer}{edgelayer}
		\draw[line width=2pt] (0.center) to (2);
	\end{pgfonlayer}
      \end{tikzpicture}}
  \end{aligned}
}

\newcommand{\idone}[1]
{
  \begin{aligned}
    \resizebox{#1}{!}{
     \begin{tikzpicture}
	\begin{pgfonlayer}{nodelayer}
		\node [style=none] (0) at (-1, -0) {};
		\node [style=none] (1) at (1, -0) {};
		\node [style=none] (2) at (0, 0.5) {};
		\node [style=none] (3) at (0, -0.5) {};
	\end{pgfonlayer}
	\begin{pgfonlayer}{edgelayer}
		\draw[line width=2pt] (1.center) to (0.center);
	\end{pgfonlayer}
\end{tikzpicture} 
    }
  \end{aligned}
}

\newcommand{\swap}[1]
{
  \begin{aligned}
    \resizebox{#1}{!}{
\begin{tikzpicture}
	\begin{pgfonlayer}{nodelayer}
		\node [style=none] (2) at (-0.5, -0.5) {};
		\node [style=none] (3) at (-2, 0.5) {};
		\node [style=none] (4) at (-0.5, 0.5) {};
		\node [style=none] (5) at (-2, -0.5) {};
	\end{pgfonlayer}
	\begin{pgfonlayer}{edgelayer}
		\draw[line width=2pt] [in=180, out=0, looseness=1.00] (3.center) to (2.center);
		\draw[line width=2pt] [in=0, out=180, looseness=1.00] (4.center) to (5.center);
	\end{pgfonlayer}
\end{tikzpicture}
    }
  \end{aligned}
}
\newcommand{\assocl}[1]
{
  \begin{aligned}
    \resizebox{#1}{!}{
\begin{tikzpicture}
	\begin{pgfonlayer}{nodelayer}
		\node [style=circ] (0) at (0.125, -0) {};
		\node [style=none] (1) at (-1, 0.5) {};
		\node [style=none] (2) at (-1, -0.5) {};
		\node [style=none] (3) at (0, -1) {};
		\node [style=none] (4) at (2.25, -0.5) {};
		\node [style=none] (5) at (0.25, -0) {};
		\node [style=circ] (6) at (1.25, -0.5) {};
		\node [style=none] (7) at (-1, -1) {};
	\end{pgfonlayer}
	\begin{pgfonlayer}{edgelayer}
		\draw[line width=2pt] [in=0, out=120, looseness=1.20] (0.center) to (1.center);
		\draw[line width=2pt] [in=0, out=-120, looseness=1.20] (0.center) to (2.center);
		\draw[line width=2pt] (4.center) to (6);
		\draw[line width=2pt] [in=0, out=120, looseness=1.20] (6) to (5.center);
		\draw[line width=2pt] [in=0, out=-120, looseness=1.20] (6) to (3.center);
		\draw[line width=2pt] (3.center) to (7.center);
	\end{pgfonlayer}
      \end{tikzpicture}}
  \end{aligned}
}

\newcommand{\assocr}[1]
{
  \begin{aligned}
    \resizebox{#1}{!}{
\begin{tikzpicture}
	\begin{pgfonlayer}{nodelayer}
		\node [style=circ] (0) at (0.125, -0.5) {};
		\node [style=none] (1) at (-1, -1) {};
		\node [style=none] (2) at (-1, 0) {};
		\node [style=none] (3) at (0, 0.5) {};
		\node [style=none] (4) at (2.25, 0) {};
		\node [style=none] (5) at (0.25, -0.5) {};
		\node [style=circ] (6) at (1.25, 0) {};
		\node [style=none] (7) at (-1, 0.5) {};
	\end{pgfonlayer}
	\begin{pgfonlayer}{edgelayer}
		\draw[line width=2pt] [in=0, out=-120, looseness=1.20] (0.center) to (1.center);
		\draw[line width=2pt] [in=0, out=120, looseness=1.20] (0.center) to (2.center);
		\draw[line width=2pt] (4.center) to (6);
		\draw[line width=2pt] [in=0, out=-120, looseness=1.20] (6) to (5.center);
		\draw[line width=2pt] [in=0, out=120, looseness=1.20] (6) to (3.center);
		\draw[line width=2pt] (3.center) to (7.center);
	\end{pgfonlayer}
      \end{tikzpicture}}
  \end{aligned}
}

\newcommand{\coassocl}[1]
{
  \begin{aligned}
    \resizebox{#1}{!}{
\begin{tikzpicture}
	\begin{pgfonlayer}{nodelayer}
		\node [style=circ] (0) at (1.125, -0.5) {};
		\node [style=none] (1) at (2.25, -1) {};
		\node [style=none] (2) at (2.25, 0) {};
		\node [style=none] (3) at (1.25, 0.5) {};
		\node [style=none] (4) at (-1, 0) {};
		\node [style=none] (5) at (1, -0.5) {};
		\node [style=circ] (6) at (0, 0) {};
		\node [style=none] (7) at (2.25, 0.5) {};
	\end{pgfonlayer}
	\begin{pgfonlayer}{edgelayer}
		\draw[line width=2pt] [in=180, out=-60, looseness=1.20] (0.center) to (1.center);
		\draw[line width=2pt] [in=180, out=60, looseness=1.20] (0.center) to (2.center);
		\draw[line width=2pt] (4.center) to (6);
		\draw[line width=2pt] [in=180, out=-60, looseness=1.20] (6) to (5.center);
		\draw[line width=2pt] [in=180, out=60, looseness=1.20] (6) to (3.center);
		\draw[line width=2pt] (3.center) to (7.center);
	\end{pgfonlayer}
      \end{tikzpicture}}
  \end{aligned}
}

\newcommand{\coassocr}[1]
{
  \begin{aligned}
    \resizebox{#1}{!}{
\begin{tikzpicture}
	\begin{pgfonlayer}{nodelayer}
		\node [style=circ] (0) at (1.125, 0) {};
		\node [style=none] (1) at (2.25, 0.5) {};
		\node [style=none] (2) at (2.25, -0.5) {};
		\node [style=none] (3) at (1.25, -1) {};
		\node [style=none] (4) at (-1, -0.5) {};
		\node [style=none] (5) at (1, 0) {};
		\node [style=circ] (6) at (0, -0.5) {};
		\node [style=none] (7) at (2.25, -1) {};
	\end{pgfonlayer}
	\begin{pgfonlayer}{edgelayer}
		\draw[line width=2pt] [in=180, out=60, looseness=1.20] (0.center) to (1.center);
		\draw[line width=2pt] [in=180, out=-60, looseness=1.20] (0.center) to (2.center);
		\draw[line width=2pt] (4.center) to (6);
		\draw[line width=2pt] [in=180, out=60, looseness=1.20] (6) to (5.center);
		\draw[line width=2pt] [in=180, out=-60, looseness=1.20] (6) to (3.center);
		\draw[line width=2pt] (3.center) to (7.center);
	\end{pgfonlayer}
      \end{tikzpicture}}
  \end{aligned}
}

\newcommand{\unitl}[1]
{
  \begin{aligned}
    \resizebox{#1}{!}{
\begin{tikzpicture}
	\begin{pgfonlayer}{nodelayer}
		\node [style=none] (0) at (1, -0) {};
		\node [style=circ] (1) at (0.125, -0) {};
		\node [style=circ] (2) at (-1, 0.5) {};
		\node [style=none] (3) at (-1, -0.5) {};
		\node [style=none] (4) at (-2, -0.5) {};
	\end{pgfonlayer}
	\begin{pgfonlayer}{edgelayer}
		\draw[line width=2pt] (0.center) to (1.center);
		\draw[line width=2pt] [in=0, out=120, looseness=1.20] (1.center) to (2.center);
		\draw[line width=2pt] [in=0, out=-120, looseness=1.20] (1.center) to (3.center);
		\draw[line width=2pt] (4.center) to (3.center);
	\end{pgfonlayer}
\end{tikzpicture}
    }
  \end{aligned}
}

\newcommand{\counitl}[1]
{
  \begin{aligned}
    \resizebox{#1}{!}{
\begin{tikzpicture}
	\begin{pgfonlayer}{nodelayer}
		\node [style=none] (0) at (-2, -0) {};
		\node [style=circ] (1) at (-1.125, -0) {};
		\node [style=circ] (2) at (0, 0.5) {};
		\node [style=none] (3) at (0, -0.5) {};
		\node [style=none] (4) at (1, -0.5) {};
	\end{pgfonlayer}
	\begin{pgfonlayer}{edgelayer}
		\draw[line width=2pt] (0.center) to (1.center);
		\draw[line width=2pt] [in=180, out=60, looseness=1.20] (1.center) to (2.center);
		\draw[line width=2pt] [in=180, out=-60, looseness=1.20] (1.center) to (3.center);
		\draw[line width=2pt] (4.center) to (3.center);
	\end{pgfonlayer}
\end{tikzpicture}
    }
  \end{aligned}
}

\newcommand{\commute}[1]
{
  \begin{aligned}
    \resizebox{#1}{!}{
\begin{tikzpicture}
	\begin{pgfonlayer}{nodelayer}
		\node [style=none] (0) at (1.25, -0) {};
		\node [style=circ] (1) at (0.375, -0) {};
		\node [style=none] (2) at (-0.5, -0.5) {};
		\node [style=none] (3) at (-2, 0.5) {};
		\node [style=none] (4) at (-0.5, 0.5) {};
		\node [style=none] (5) at (-2, -0.5) {};
	\end{pgfonlayer}
	\begin{pgfonlayer}{edgelayer}
		\draw[line width=2pt] (0.center) to (1.center);
		\draw[line width=2pt] [in=0, out=-120, looseness=1.20] (1.center) to (2.center);
		\draw[line width=2pt] [in=180, out=0, looseness=1.00] (3.center) to (2.center);
		\draw[line width=2pt] [in=0, out=120, looseness=1.20] (1.center) to (4.center);
		\draw[line width=2pt] [in=0, out=180, looseness=1.00] (4.center) to (5.center);
	\end{pgfonlayer}
\end{tikzpicture}
    }
  \end{aligned}
}

\newcommand{\cocommute}[1]
{
  \begin{aligned}
    \resizebox{#1}{!}{
\begin{tikzpicture}
	\begin{pgfonlayer}{nodelayer}
		\node [style=none] (0) at (-2, -0) {};
		\node [style=circ] (1) at (-1.125, -0) {};
		\node [style=none] (2) at (-0.25, -0.5) {};
		\node [style=none] (3) at (1.25, 0.5) {};
		\node [style=none] (4) at (-0.25, 0.5) {};
		\node [style=none] (5) at (1.25, -0.5) {};
	\end{pgfonlayer}
	\begin{pgfonlayer}{edgelayer}
		\draw[line width=2pt] (0.center) to (1.center);
		\draw[line width=2pt] [in=180, out=-60, looseness=1.20] (1.center) to (2.center);
		\draw[line width=2pt] [in=0, out=180, looseness=1.00] (3.center) to (2.center);
		\draw[line width=2pt] [in=180, out=60, looseness=1.20] (1.center) to (4.center);
		\draw[line width=2pt] [in=180, out=0, looseness=1.00] (4.center) to (5.center);
	\end{pgfonlayer}
\end{tikzpicture}
    }
  \end{aligned}
}

\newcommand{\frobs}[1]
{
  \begin{aligned}
    \resizebox{#1}{!}{
\begin{tikzpicture}
	\begin{pgfonlayer}{nodelayer}
		\node [style=none] (0) at (-1.5, 0.5) {};
		\node [style=circ] (1) at (-0.75, 0.5) {};
		\node [style=none] (2) at (0.25, -0) {};
		\node [style=none] (3) at (0.25, 1) {};
		\node [style=circ] (4) at (1, -0.5) {};
		\node [style=none] (5) at (0, -0) {};
		\node [style=none] (6) at (1.75, -0.5) {};
		\node [style=none] (7) at (0, -1) {};
		\node [style=none] (8) at (1.75, 1) {};
		\node [style=none] (9) at (-1.5, -1) {};
	\end{pgfonlayer}
	\begin{pgfonlayer}{edgelayer}
		\draw[line width=2pt] [in=180, out=-60, looseness=1.20] (1) to (2.center);
		\draw[line width=2pt] [in=180, out=60, looseness=1.20] (1) to (3.center);
		\draw[line width=2pt] (0.center) to (1);
		\draw[line width=2pt] (6.center) to (4);
		\draw[line width=2pt] [in=0, out=120, looseness=1.20] (4) to (5.center);
		\draw[line width=2pt] [in=0, out=-120, looseness=1.20] (4) to (7.center);
		\draw[line width=2pt] (3.center) to (8.center);
		\draw[line width=2pt] (7.center) to (9.center);
	\end{pgfonlayer}
\end{tikzpicture}
    }
  \end{aligned}
}

\newcommand{\frobx}[1]
{
  \begin{aligned}
    \resizebox{#1}{!}{
\begin{tikzpicture}
	\begin{pgfonlayer}{nodelayer}
		\node [style=circ] (0) at (-0.5, -0) {};
		\node [style=none] (1) at (-1.5, -0.5) {};
		\node [style=none] (2) at (-1.5, 0.5) {};
		\node [style=circ] (3) at (0.5, -0) {};
		\node [style=none] (4) at (1.5, -0.5) {};
		\node [style=none] (5) at (1.5, 0.5) {};
	\end{pgfonlayer}
	\begin{pgfonlayer}{edgelayer}
		\draw[line width=2pt] [in=0, out=-120, looseness=1.20] (0.center) to (1.center);
		\draw[line width=2pt] [in=0, out=120, looseness=1.20] (0.center) to (2.center);
		\draw[line width=2pt] [in=180, out=-60, looseness=1.20] (3) to (4.center);
		\draw[line width=2pt] [in=180, out=60, looseness=1.20] (3) to (5.center);
		\draw[line width=2pt] (0) to (3);
	\end{pgfonlayer}
\end{tikzpicture}
    }
  \end{aligned}
}

\newcommand{\frobz}[1]
{
  \begin{aligned}
    \resizebox{#1}{!}{
\begin{tikzpicture}
	\begin{pgfonlayer}{nodelayer}
		\node [style=none] (0) at (1.75, 0.5) {};
		\node [style=circ] (1) at (1, 0.5) {};
		\node [style=none] (2) at (0, -0) {};
		\node [style=none] (3) at (0, 1) {};
		\node [style=circ] (4) at (-0.75, -0.5) {};
		\node [style=none] (5) at (0.25, -0) {};
		\node [style=none] (6) at (-1.5, -0.5) {};
		\node [style=none] (7) at (0.25, -1) {};
		\node [style=none] (8) at (-1.5, 1) {};
		\node [style=none] (9) at (1.75, -1) {};
	\end{pgfonlayer}
	\begin{pgfonlayer}{edgelayer}
		\draw[line width=2pt] [in=0, out=-120, looseness=1.20] (1) to (2.center);
		\draw[line width=2pt] [in=0, out=120, looseness=1.20] (1) to (3.center);
		\draw[line width=2pt] (0.center) to (1);
		\draw[line width=2pt] (6.center) to (4);
		\draw[line width=2pt] [in=180, out=60, looseness=1.20] (4) to (5.center);
		\draw[line width=2pt] [in=180, out=-60, looseness=1.20] (4) to (7.center);
		\draw[line width=2pt] (3.center) to (8.center);
		\draw[line width=2pt] (7.center) to (9.center);
	\end{pgfonlayer}
\end{tikzpicture}
    }
  \end{aligned}
}

\newcommand{\bi}[1]
{
  \begin{aligned}
    \resizebox{#1}{!}{
\begin{tikzpicture}
	\begin{pgfonlayer}{nodelayer}
		\node [style=none] (0) at (-2, 0.75) {};
		\node [style=none] (1) at (-2, -0.5) {};
		\node [style=circ] (2) at (-1, 0.75) {};
		\node [style=circ] (3) at (-1, -0.5) {};
		\node [style=none] (4) at (0, 1.25) {};
		\node [style=none] (5) at (-0.25, 0.25) {};
		\node [style=none] (6) at (-0.25, -0) {};
		\node [style=none] (7) at (0, -1) {};
		\node [style=none] (8) at (0, 1.25) {};
		\node [style=none] (9) at (0.25, 0.25) {};
		\node [style=none] (10) at (0.25, -0) {};
		\node [style=none] (11) at (0, -1) {};
		\node [style=circ] (12) at (1, 0.75) {};
		\node [style=circ] (13) at (1, -0.5) {};
		\node [style=none] (14) at (2, 0.75) {};
		\node [style=none] (15) at (2, -0.5) {};
	\end{pgfonlayer}
	\begin{pgfonlayer}{edgelayer}
		\draw[line width=2pt] (0.center) to (2);
		\draw[line width=2pt] (1.center) to (3);
		\draw[line width=2pt] (12) to (14.center);
		\draw[line width=2pt] (13) to (15.center);
		\draw[line width=2pt] [in=180, out=60, looseness=1.20] (2) to (4.center);
		\draw[line width=2pt] [in=180, out=-60, looseness=1.20] (2) to (5.center);
		\draw[line width=2pt] [in=180, out=60, looseness=1.20] (3) to (6.center);
		\draw[line width=2pt] [in=180, out=-60, looseness=1.20] (3) to (7.center);
		\draw[line width=2pt] [in=120, out=0, looseness=1.20] (8.center) to (12);
		\draw[line width=2pt] [in=-120, out=0, looseness=1.20] (9.center) to (12);
		\draw[line width=2pt] [in=120, out=0, looseness=1.20] (10.center) to (13);
		\draw[line width=2pt] [in=-120, out=0, looseness=1.20] (11.center) to (13);
		\draw[line width=2pt] (4.center) to (8.center);
		\draw[line width=2pt] [in=150, out=-30, looseness=1.00] (5.center) to (10.center);
		\draw[line width=2pt] [in=-150, out=30, looseness=1.00] (6.center) to (9.center);
		\draw[line width=2pt] (7.center) to (11.center);
	\end{pgfonlayer}
\end{tikzpicture}
    }
  \end{aligned}
}

\newcommand{\bimultl}[1]
{
\begin{aligned}
    \resizebox{#1}{!}{
\begin{tikzpicture}
	\begin{pgfonlayer}{nodelayer}
		\node [style=circ] (0) at (1, -0) {};
		\node [style=circ] (1) at (0.125, -0) {};
		\node [style=none] (2) at (-1, 0.5) {};
		\node [style=none] (3) at (-1, -0.5) {};
	\end{pgfonlayer}
	\begin{pgfonlayer}{edgelayer}
		\draw[line width=2pt] (0.center) to (1.center);
		\draw[line width=2pt] [in=0, out=120, looseness=1.20] (1.center) to (2.center);
		\draw[line width=2pt] [in=0, out=-120, looseness=1.20] (1.center) to (3.center);
	\end{pgfonlayer}
      \end{tikzpicture}}
\end{aligned}
}

\newcommand{\bimultr}[1]
{
\begin{aligned}
    \resizebox{#1}{!}{
\begin{tikzpicture}
	\begin{pgfonlayer}{nodelayer}
		\node [style=none] (e) at (1, -0) {};
		\node [style=circ] (0) at (0.125, 0.5) {};
		\node [style=circ] (1) at (0.125, -0.5) {};
		\node [style=none] (2) at (-1, 0.5) {};
		\node [style=none] (3) at (-1, -0.5) {};
	\end{pgfonlayer}
	\begin{pgfonlayer}{edgelayer}
		\draw[line width=2pt] (2.center) to (0.center);
		\draw[line width=2pt] (3.center) to (1.center);
	\end{pgfonlayer}
      \end{tikzpicture}}
\end{aligned}
}

\newcommand{\bicomultl}[1]
{
\begin{aligned}
    \resizebox{#1}{!}{
\begin{tikzpicture}
	\begin{pgfonlayer}{nodelayer}
		\node [style=circ] (0) at (-1, -0) {};
		\node [style=circ] (1) at (-0.125, -0) {};
		\node [style=none] (2) at (1, 0.5) {};
		\node [style=none] (3) at (1, -0.5) {};
	\end{pgfonlayer}
	\begin{pgfonlayer}{edgelayer}
		\draw[line width=2pt] (0.center) to (1.center);
		\draw[line width=2pt] [in=180, out=60, looseness=1.20] (1.center) to (2.center);
		\draw[line width=2pt] [in=180, out=-60, looseness=1.20] (1.center) to (3.center);
	\end{pgfonlayer}
      \end{tikzpicture}}
\end{aligned}
}

\newcommand{\bicomultr}[1]
{
\begin{aligned}
    \resizebox{#1}{!}{
\begin{tikzpicture}
	\begin{pgfonlayer}{nodelayer}
		\node [style=none] (e) at (-1, -0) {};
		\node [style=circ] (0) at (-0.125, 0.5) {};
		\node [style=circ] (1) at (-0.125, -0.5) {};
		\node [style=none] (2) at (1, 0.5) {};
		\node [style=none] (3) at (1, -0.5) {};
	\end{pgfonlayer}
	\begin{pgfonlayer}{edgelayer}
		\draw[line width=2pt] (0.center) to (2.center);
		\draw[line width=2pt] (1.center) to (3.center);
	\end{pgfonlayer}
      \end{tikzpicture}}
\end{aligned}
}

\newcommand{\spec}[1]
{
  \begin{aligned}
    \resizebox{#1}{!}{
\begin{tikzpicture}
	\begin{pgfonlayer}{nodelayer}
		\node [style=none] (0) at (1.75, -0) {};
		\node [style=circ] (1) at (0.75, -0) {};
		\node [style=none] (2) at (0, -0.5) {};
		\node [style=none] (3) at (0, 0.5) {};
		\node [style=circ] (4) at (-0.75, -0) {};
		\node [style=none] (5) at (0, -0.5) {};
		\node [style=none] (6) at (-1.75, -0) {};
		\node [style=none] (7) at (0, 0.5) {};
	\end{pgfonlayer}
	\begin{pgfonlayer}{edgelayer}
		\draw[line width=2pt] (0.center) to (1.center);
		\draw[line width=2pt] [in=0, out=-120, looseness=1.20] (1.center) to (2.center);
		\draw[line width=2pt] [in=0, out=120, looseness=1.20] (1.center) to (3.center);
		\draw[line width=2pt] (6.center) to (4);
		\draw[line width=2pt] [in=180, out=-60, looseness=1.20] (4) to (5.center);
		\draw[line width=2pt] [in=180, out=60, looseness=1.20] (4) to (7.center);
	\end{pgfonlayer}
\end{tikzpicture}
    }
  \end{aligned}
}

\newcommand{\extral}[1]
{
  \begin{aligned}
    \resizebox{#1}{!}{
\begin{tikzpicture}
	\begin{pgfonlayer}{nodelayer}
		\node [style=none] (0) at (1.75, -0) {};
		\node [style=circ] (1) at (0.75, -0) {};
		\node [style=circ] (4) at (-0.75, -0) {};
		\node [style=none] (6) at (-1.75, -0) {};
	\end{pgfonlayer}
	\begin{pgfonlayer}{edgelayer}
	  \draw[line width=2pt] (1.center) to (4.center);
	\end{pgfonlayer}
\end{tikzpicture}
    }
  \end{aligned}
}

\newcommand{\extrar}[1]
{
  \begin{aligned}
    \resizebox{#1}{!}{
\begin{tikzpicture}
	\begin{pgfonlayer}{nodelayer}
		\node [style=none] (0) at (1.75, -0) {};
		\node [style=none] (6) at (-1.75, -0) {};
	\end{pgfonlayer}
\end{tikzpicture}
    }
  \end{aligned}
}

\title{Corelations are the prop for extraspecial commutative Frobenius monoids}
\author{Brandon Coya and Brendan Fong}

  \thanks{We thank John Baez for many useful discussions and advice, and Jason
    Erbele for comments on a draft. BF thanks the Clarendon Fund, Hertford
    College, and the Queen Elizabeth Scholarships, Oxford for their support.
  }

  \address{\vspace{-7ex}
    \begin{multicols}{2}
    Department of Computer Science \\
    University of Oxford \\
    United Kingdom OX1 3QD
\columnbreak \\ 
    Department of Mathematics \\
    University of Pennsylvania \\
    USA 19104
\end{multicols}
    Department of Mathematics \\
    University of California Riverside \\
    USA 92521
  }

  \copyrightyear{2015}

  \eaddress{brendan.fong@cs.ox.ac.uk, bcoya001@ucr.edu}

\begin{document}   

\maketitle

\begin{abstract}
  Just as binary relations between sets may be understood as jointly monic
  spans, so too may equivalence relations on the disjoint union of sets be
  understood as jointly epic cospans. With the ensuing notion of composition
  inherited from the pushout of cospans, we call these equivalence relations
  \emph{corelations}. We define the category of corelations between finite sets
  and prove that it is equivalent to the prop for extraspecial commutative
  Frobenius monoids.  Dually, we show that the category of relations is
  equivalent to the prop for special commutative bimonoids. Throughout, we
  emphasise how corelations model interconnection. 
\end{abstract}

\section{Introduction}
It is well-known that the category of relations between finite sets may
be obtained as the category of isomorphism classes of jointly monic spans in the
category of finite sets and functions. In this paper we investigate the dual
notion: isomorphism classes of jointly epic cospans. These are known as
corelations, and corelations from a set $X$ to a set $Y$ are characterised as
partitions of the disjoint union $X+Y$.

Our slogan is `corelations model connection'. We understand a corelation as a
partition of two sets into connected components, depicting examples as follows
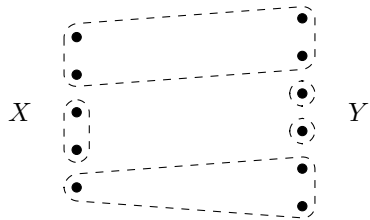
\[
  \begin{tikzpicture}[circuit ee IEC]
	\begin{pgfonlayer}{nodelayer}
		\node [contact, outer sep=5pt] (0) at (-2, 1) {};
		\node [contact, outer sep=5pt] (1) at (-2, 0.5) {};
		\node [contact, outer sep=5pt] (2) at (-2, -0) {};
		\node [contact, outer sep=5pt] (3) at (-2, -0.5) {};
		\node [contact, outer sep=5pt] (4) at (-2, -1) {};
		\node [contact, outer sep=5pt] (5) at (1, 1.25) {};
		\node [contact, outer sep=5pt] (6) at (1, 0.75) {};
		\node [contact, outer sep=5pt] (7) at (1, 0.25) {};
		\node [contact, outer sep=5pt] (8) at (1, -0.25) {};
		\node [contact, outer sep=5pt] (9) at (1, -0.75) {};
		\node [contact, outer sep=5pt] (10) at (1, -1.25) {};
		\node [style=none] (11) at (-2.75, -0) {$X$};
		\node [style=none] (12) at (1.75, -0) {$Y$};
	\end{pgfonlayer}
	\begin{pgfonlayer}{edgelayer}
		\draw [rounded corners=5pt, dashed] 
   (node cs:name=0, anchor=north west) --
   (node cs:name=1, anchor=south west) --
   (node cs:name=6, anchor=south east) --
   (node cs:name=5, anchor=north east) --
   cycle;
		\draw [rounded corners=5pt, dashed] 
   (node cs:name=2, anchor=north west) --
   (node cs:name=3, anchor=south west) --
   (node cs:name=3, anchor=south east) --
   (node cs:name=2, anchor=north east) --
   cycle;
		\draw [rounded corners=5pt, dashed] 
   (node cs:name=4, anchor=north west) --
   (node cs:name=4, anchor=south west) --
   (node cs:name=10, anchor=south east) --
   (node cs:name=9, anchor=north east) --
   cycle;
   		\draw [rounded corners=5pt, dashed] 
   (node cs:name=7, anchor=north west) --
   (node cs:name=7, anchor=south west) --
   (node cs:name=7, anchor=south east) --
   (node cs:name=7, anchor=north east) --
   cycle;
   		\draw [rounded corners=5pt, dashed] 
   (node cs:name=8, anchor=north west) --
   (node cs:name=8, anchor=south west) --
   (node cs:name=8, anchor=south east) --
   (node cs:name=8, anchor=north east) --
   cycle;
	\end{pgfonlayer}
\end{tikzpicture}
\]
Here we have a corelation from a set $X$ of five elements to a set $Y$ of six
elements. Elements belonging to the same equivalence class of $X+Y$ are grouped
(`connected') by a dashed line.

Composition of corelations takes the transitive closure of the two partitions,
before restricting the partition to the new domain and codomain. For example,
suppose in addition to the corelation $\alpha\maps X \to Y$ above we have
another corelation $\beta\maps Y \to Z$
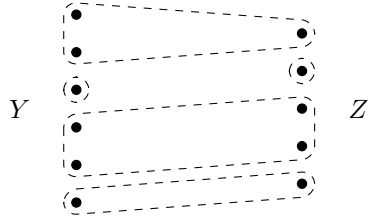
\[
\begin{tikzpicture}[circuit ee IEC]
	\begin{pgfonlayer}{nodelayer}
		\node [style=none] (0) at (-2.75, -0) {$Y$};
		\node [style=none] (1) at (1.75, 0) {$Z$};
		\node [contact, outer sep=5pt] (2) at (-2, 1.25) {};
		\node [contact, outer sep=5pt] (3) at (-2, 0.75) {};
		\node [contact, outer sep=5pt] (4) at (-2, 0.25) {};
		\node [contact, outer sep=5pt] (5) at (-2, -0.25) {};
		\node [contact, outer sep=5pt] (6) at (-2, -0.75) {};
		\node [contact, outer sep=5pt] (7) at (-2, -1.25) {};
		\node [contact, outer sep=5pt] (8) at (1, 1) {};
		\node [contact, outer sep=5pt] (9) at (1, 0.5) {};
		\node [contact, outer sep=5pt] (10) at (1, -0) {};
		\node [contact, outer sep=5pt] (11) at (1, -0.5) {};
		\node [contact, outer sep=5pt] (12) at (1, -1) {};
	\end{pgfonlayer}
		\draw [rounded corners=5pt, dashed] 
   (node cs:name=2, anchor=north west) --
   (node cs:name=3, anchor=south west) --
   (node cs:name=8, anchor=south east) --
   (node cs:name=8, anchor=north east) --
   cycle;
		\draw [rounded corners=5pt, dashed] 
   (node cs:name=4, anchor=north west) --
   (node cs:name=4, anchor=south west) --
   (node cs:name=4, anchor=south east) --
   (node cs:name=4, anchor=north east) --
   cycle;
		\draw [rounded corners=5pt, dashed] 
   (node cs:name=5, anchor=north west) --
   (node cs:name=6, anchor=south west) --
   (node cs:name=11, anchor=south east) --
   (node cs:name=10, anchor=north east) --
   cycle;
		\draw [rounded corners=5pt, dashed] 
   (node cs:name=7, anchor=north west) --
   (node cs:name=7, anchor=south west) --
   (node cs:name=12, anchor=south east) --
   (node cs:name=12, anchor=north east) --
   cycle;
		\draw [rounded corners=5pt, dashed] 
   (node cs:name=9, anchor=north west) --
   (node cs:name=9, anchor=south west) --
   (node cs:name=9, anchor=south east) --
   (node cs:name=9, anchor=north east) --
   cycle;
\end{tikzpicture}
\]
Then the composite $\beta\circ\alpha$ of our two corelations is given by
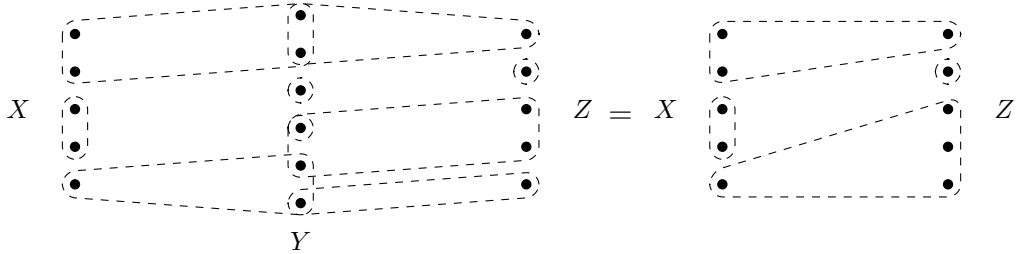
\vspace{-1ex}
\[
  \begin{aligned}
\begin{tikzpicture}[circuit ee IEC]
	\begin{pgfonlayer}{nodelayer}
		\node [contact, outer sep=5pt] (-2) at (1, 1.25) {};
		\node [contact, outer sep=5pt] (-1) at (1, 0.75) {};
		\node [contact, outer sep=5pt] (0) at (1, 0.25) {};
		\node [contact, outer sep=5pt] (1) at (1, -0.25) {};
		\node [contact, outer sep=5pt] (2) at (1, -0.75) {};
		\node [contact, outer sep=5pt] (3) at (1, -1.25) {};
		\node [style=none] (4) at (-2.75, -0) {$X$};
		\node [style=none] (5) at (4.75, -0) {$Z$};
		\node [contact, outer sep=5pt] (6) at (-2, 1) {};
		\node [contact, outer sep=5pt] (7) at (-2, -0.5) {};
		\node [contact, outer sep=5pt] (8) at (-2, 0.5) {};
		\node [contact, outer sep=5pt] (9) at (-2, -0) {};
		\node [contact, outer sep=5pt] (10) at (-2, -1) {};
		\node [contact, outer sep=5pt] (11) at (4, -0) {};
		\node [contact, outer sep=5pt] (12) at (4, -1) {};
		\node [contact, outer sep=5pt] (13) at (4, -0.5) {};
		\node [contact, outer sep=5pt] (14) at (4, 0.5) {};
		\node [contact, outer sep=5pt] (19) at (4, 1) {};
		\node [style=none] (20) at (1, -1.75) {$Y$};
		\node [style=none] (21) at (1, 1.75) {\phantom{$Y$}};
	\end{pgfonlayer}
	\begin{pgfonlayer}{edgelayer}
		\draw [rounded corners=5pt, dashed] 
   (node cs:name=6, anchor=north west) --
   (node cs:name=8, anchor=south west) --
   (node cs:name=-1, anchor=south east) --
   (node cs:name=-2, anchor=north east) --
   cycle;
		\draw [rounded corners=5pt, dashed] 
   (node cs:name=9, anchor=north west) --
   (node cs:name=7, anchor=south west) --
   (node cs:name=7, anchor=south east) --
   (node cs:name=9, anchor=north east) --
   cycle;
		\draw [rounded corners=5pt, dashed] 
   (node cs:name=10, anchor=north west) --
   (node cs:name=10, anchor=south west) --
   (node cs:name=3, anchor=south east) --
   (node cs:name=2, anchor=north east) --
   cycle;
		\draw [rounded corners=5pt, dashed] 
   (node cs:name=-2, anchor=north west) --
   (node cs:name=-1, anchor=south west) --
   (node cs:name=19, anchor=south east) --
   (node cs:name=19, anchor=north east) --
   cycle;
		\draw [rounded corners=5pt, dashed] 
   (node cs:name=0, anchor=north west) --
   (node cs:name=0, anchor=south west) --
   (node cs:name=0, anchor=south east) --
   (node cs:name=0, anchor=north east) --
   cycle;
		\draw [rounded corners=5pt, dashed] 
   (node cs:name=1, anchor=north west) --
   (node cs:name=1, anchor=south west) --
   (node cs:name=1, anchor=south east) --
   (node cs:name=1, anchor=north east) --
   cycle;
		\draw [rounded corners=5pt, dashed] 
   (node cs:name=1, anchor=north west) --
   (node cs:name=2, anchor=south west) --
   (node cs:name=13, anchor=south east) --
   (node cs:name=11, anchor=north east) --
   cycle;
		\draw [rounded corners=5pt, dashed] 
   (node cs:name=3, anchor=north west) --
   (node cs:name=3, anchor=south west) --
   (node cs:name=12, anchor=south east) --
   (node cs:name=12, anchor=north east) --
   cycle;
		\draw [rounded corners=5pt, dashed] 
   (node cs:name=14, anchor=north west) --
   (node cs:name=14, anchor=south west) --
   (node cs:name=14, anchor=south east) --
   (node cs:name=14, anchor=north east) --
   cycle;
	\end{pgfonlayer}
\end{tikzpicture}
\end{aligned}
\:
  =
\:
\begin{aligned}
\begin{tikzpicture}[circuit ee IEC]
	\begin{pgfonlayer}{nodelayer}
		\node [style=none] (0) at (-2.75, -0) {$X$};
		\node [style=none] (1) at (1.75, -0) {$Z$};
		\node [contact, outer sep=5pt] (2) at (-2, 1) {};
		\node [contact, outer sep=5pt] (3) at (-2, -0.5) {};
		\node [contact, outer sep=5pt] (4) at (-2, 0.5) {};
		\node [contact, outer sep=5pt] (5) at (-2, -0) {};
		\node [contact, outer sep=5pt] (6) at (-2, -1) {};
		\node [contact, outer sep=5pt] (7) at (1, -0) {};
		\node [contact, outer sep=5pt] (8) at (1, -1) {};
		\node [contact, outer sep=5pt] (9) at (1, -0.5) {};
		\node [contact, outer sep=5pt] (10) at (1, 0.5) {};
		\node [contact, outer sep=5pt] (13) at (1, 1) {};
		\node [style=none] (20) at (1, -1.75) {\phantom{$Y$}};
		\node [style=none] (21) at (1, 1.75) {\phantom{$Y$}};
	\end{pgfonlayer}
	\begin{pgfonlayer}{edgelayer}
		\draw [rounded corners=5pt, dashed] 
   (node cs:name=2, anchor=north west) --
   (node cs:name=4, anchor=south west) --
   (node cs:name=13, anchor=south east) --
   (node cs:name=13, anchor=north east) --
   cycle;
		\draw [rounded corners=5pt, dashed] 
   (node cs:name=5, anchor=north west) --
   (node cs:name=3, anchor=south west) --
   (node cs:name=3, anchor=south east) --
   (node cs:name=5, anchor=north east) --
   cycle;
		\draw [rounded corners=5pt, dashed] 
   (node cs:name=6, anchor=north west) --
   (node cs:name=6, anchor=south west) --
   (node cs:name=8, anchor=south east) --
   (node cs:name=7, anchor=north east) --
   cycle;
		\draw [rounded corners=5pt, dashed] 
   (node cs:name=10, anchor=north west) --
   (node cs:name=10, anchor=south west) --
   (node cs:name=10, anchor=south east) --
   (node cs:name=10, anchor=north east) --
   cycle;
	\end{pgfonlayer}
\end{tikzpicture}
\end{aligned}
\]
Informally, this captures the idea that two elements of $X+Z$ are `connected' if
we may travel from one to the other staying within connected components of
$\alpha$ and $\beta$.

Another structure that axiomatises interconnection is the extraspecial
commutative Frobenius monoid. An extraspecial commutative Frobenius monoid in a
symmetric monoidal category is an object equipped with commutative monoid and
cocommutative comonoid structures obeying additional laws known as the
Frobenius, special, and extra laws.  Special commutative Frobenius monoids are
well-known; the additional axiom here, the so-called extra law, requires that
the unit composed with the counit is the identity on the unit for the monoidal
product. We write this in string diagrams as
\[
  \extral{.15\textwidth} = \extrar{.15\textwidth}
\]
Together, these axioms express the idea that connectivity is all that matters:
not pairwise clustering, not multiple paths, not `extra', interior components.

Corelations and extraspecial commutative Frobenius monoids are intimately
related. To explicate this relationship, we will use the language of props.
Recall that a prop\footnote{Often stylised PROP, for PROduct and Permutation
category.} is a strict symmetric monoidal category with objects the natural
numbers and monoidal product addition. Also recall that a prop $\mathcal T$ is
termed the prop for an algebraic structure if, given another symmetric monoidal
category $\mathcal{C}$, the strict symmetric monoidal functor category
$\mathcal{C}^{\mathcal T}$ is isomorphic to the category of the chosen algebraic
structure in $\mathcal{C}$. 

Considered as symmetric monoidal categories with monoidal product the disjoint
union, Lack proved that the category of spans in the category of finite sets and
functions is equivalent as a symmetric monoidal category to the prop for
bicommutative bimonoids, and the category of cospans is equivalent as a
symmetric monoidal category to the prop for special commutative Frobenius
monoids \cite{La}. Note that the disjoint union of finite sets also gives a
monoidal product on the category of corelations. Our main theorem is:
\begin{theorem}
   The category of corelations is equivalent, as a symmetric monoidal category,
   to the prop for extraspecial commutative Frobenius monoids.
\end{theorem}

Corelations and extraspecial commutative Frobenius monoids have been observed to
play a key role in many frameworks relying on the interconnection of systems,
including electrical circuits \cite{BF}, signal flow graphs \cite{BE,BSZ}, bond
graphs \cite{BC}, linear time-invariant systems \cite{FRS}, automata \cite{RSW},
proofs \cite{DP1}, and matrices and other linear systems \cite{Za}. The mutual
characterisation of these structures provided by our main theorem clarifies and
streamlines arguments in many of these applications. For example, the use of
corelations provides a precise extraspecial commutative Frobenius monoid
extension of the well-known `spider theorem' characterising morphisms between
tensor powers of a special commutative Frobenius monoid \cite{CK,CPP}.

In independent but related work, Zanasi proves in his recent thesis
\cite[\textsection 2.5]{Za} that the so-called prop of equivalence relations is
the free prop on the theory of extraspecial commutative Frobenius monoids. As we
do, Zanasi builds on Lack's observation that category of cospans in the category
of finite sets and functions is equivalent to the prop for special commutative
Frobenius monoids \cite{La}, as well as the observation of Bruni and Gadducci
that cospans are related to equivalence relations \cite{BG}. Zanasi argues via a
so-called `cube construction', or fibred sum of props. 

Similarly, Do\v{s}en and Petri\'c \cite[\textsection 9]{DP2} prove that the
category of `split equivalences' is isomorphic to the `equivalential Frobenius
monad freely generated by a single object'. They argue this analogous result by
constructing an auxiliary syntactic category isomorphic to the equivalential
Frobenius monad freely generated by a single object, and then inducting on the
terms of this new category to prove the main result.

Our novel approach through the understanding of corelations as jointly epic
cospans permits a significantly simpler argument via a coequalizer of props. In
doing so, it provides a clear narrative for the origin of the extra law and its
relationship with other fundamental axioms. Moreover, such an approach is
philosophically well-motivated, and provides easy generalisation, such as the
characterisation of linear relations as jointly epic cospans in the category of
matrices over a field \cite{Fo}, or linear time-invariant systems as jointly
epic cospans in the category of matrices over a relevant Laurent polynomial ring
\cite{FRS}.

Ultimately, our work completes the beautiful picture
\[
\begin{tabular}{c|c}
  \textbf{spans} & \textbf{cospans} \\
  bicommutative bimonoids & special commutative Frobenius monoids \\ \hline
  \textbf{relations} & \textbf{corelations} \\
  special bicommutative bimonoids & extraspecial commutative Frobenius monoids \\
\end{tabular}
\]
pairing constructions on the category of finite sets and functions with
important algebraic structures. The duality of the bimonoid laws and the
Frobenius law, the two major ways that a monoid and comonoid can interact, was
demonstrated by Lack \cite{La}. Moving this duality to the level of relations,
we show that the importance of the heretofore overlooked extra law as the dual
version of the special law.

\subsection{Outline} In the next two sections we introduce corelations and
extraspecial commutative Frobenius monoids respectively. These are the stars of
this paper, and our task will be to understand their relationship. To this end,
in Section \ref{sec.props} we review the idea of a prop for an algebraic
structure, and note that the category of corelations between finite sets is
equivalent to a prop $\corel$. In Section \ref{sec.theories}, we then construct
a prop $\mathbf{Th(ESCFM)}$ whose algebras are extraspecial commutative
Frobenius monoids. We show in Section~\ref{sec.mainthm} that $\corel$ and
$\mathbf{Th(ESCFM)}$ are isomorphic, proving the main theorem.  Finally, in
Section \ref{sec.summary} we outline the dual characterisation of the category
of relations, and summarise the algebraic theories corresponding to spans,
cospans, relations, and corelations. 

\section{Corelations} \label{sec.corel}

First we define corelations. Corelations arise as the dual of relations: recall
that a binary relation from a set $X$ to a set $Y$ is a subset of the product $X
\times Y$. A corelation is a quotient of the coproduct $X+Y$.

\begin{definition}
  A \define{corelation} $\alpha\maps X \to Y$ between sets $X$ and $Y$ is a
  partition of $X+Y$.

  Given another corelation $\beta\maps Y \to Z$, the composite $\beta \circ
  \alpha \maps X \to Z$ is the restriction to $X+Z$ of the finest partition on
  $X+Y+Z$ that is coarser than both $\alpha$ and $\beta$ when restricted to
  $X+Y$ and $Y+Z$ respectively.
\end{definition}

This composition is associative as both pairwise methods of composing
corelations $\alpha: X \to Y$, $\beta: Y \to Z$, and $\gamma: Z \to W$ amount to
finding the finest partition on $X+Y+Z+W$ that is coarser than each of $\alpha$,
$\beta$, and $\gamma$ when restricted to the relevant subset, and then
restricting this partition to a partition on $X+W$; reference to the motivating
visualization makes this clear. Moreover, this composition has an identity: it
is the partition of $X+X$ such that each equivalence class comprises exactly two
elements, an element $x \in X$ considered as an element of both the first and
then the second summand of $X+X$. 

This allows us to define a category. We restrict our attention to corelations
between finite sets.

\begin{definition}
  Let $\corel$ be the symmetric monoidal category with objects finite sets,
  morphisms corelations between finite sets, and monoidal product disjoint
  union.
\end{definition}

We shall freely abuse the notation $\corel$ to refer to an equivalent
skeleton. This is key for our main theorem: in Section~\ref{sec.props} we
show this skeleton is strict, and hence a prop.

Ellerman gives a detailed treatment of corelations from a logic viewpoint in
\cite{El}, while basic category theoretic aspects can be found in Lawvere and
Rosebrugh \cite{LR}. Note that neither binary relations nor corelations are a
generalisation of the other. A key property of corelation is that it forms a
compact category with monoidal product disjoint union of sets. This is not true
of the category of relations.

Another way of visualising corelations and their composition is as terminals
connected by junctions of ideal wires. We draw these by marking each equivalence
class with a point (the `junction'), and then connecting each element of the
domain and codomain to their equivalence class with a `wire'. Composition then
involves collapsing connected junctions down to a point. The example from the
introduction is represented as follows.
\vspace{-1ex}
\[
  \begin{aligned}
\begin{tikzpicture}[circuit ee IEC]
	\begin{pgfonlayer}{nodelayer}
		\node [contact, outer sep=5pt] (6) at (-2, 1) {};
		\node [contact, outer sep=5pt] (7) at (-2, -0.5) {};
		\node [contact, outer sep=5pt] (8) at (-2, 0.5) {};
		\node [contact, outer sep=5pt] (9) at (-2, -0) {};
		\node [contact, outer sep=5pt] (10) at (-2, -1) {};
		\node [style=none] (15) at (-0.5, 0.875) {};
		\node [style=none] (28) at (-0.5, 0.25) {};
		\node [style=none] (16) at (-0.5, -0.125) {};
		\node [style=none] (29) at (-0.5, -0.375) {};
		\node [style=none] (17) at (-0.5, -1) {};
		\node [contact, outer sep=5pt] (-2) at (1, 1.25) {};
		\node [contact, outer sep=5pt] (-1) at (1, 0.75) {};
		\node [contact, outer sep=5pt] (0) at (1, 0.25) {};
		\node [contact, outer sep=5pt] (1) at (1, -0.25) {};
		\node [contact, outer sep=5pt] (2) at (1, -0.75) {};
		\node [contact, outer sep=5pt] (3) at (1, -1.25) {};
		\node [style=none] (18) at (2.5, -1.125) {};
		\node [style=none] (21) at (2.5, 1) {};
		\node [style=none] (22) at (2.5, -0.375) {};
		\node [style=none] (23) at (2.5, 0.475) {};
		\node [style=none] (24) at (2.5, 0.25) {};
		\node [contact, outer sep=5pt] (19) at (4, 1) {};
		\node [contact, outer sep=5pt] (14) at (4, 0.5) {};
		\node [contact, outer sep=5pt] (11) at (4, -0) {};
		\node [contact, outer sep=5pt] (13) at (4, -0.5) {};
		\node [contact, outer sep=5pt] (12) at (4, -1) {};
		\node [style=none] (4) at (-2.75, -0) {$X$};
		\node [style=none] (5) at (4.75, -0) {$Z$};
		\node [style=none] (20) at (1, -1.75) {$Y$};
		\node [style=none] (30) at (1, 1.75) {\phantom{$Y$}};
	\end{pgfonlayer}
	\begin{pgfonlayer}{edgelayer}
		\draw [thick] (6.center) to (15.center);
		\draw [thick] (8.center) to (15.center);
		\draw [thick] (-2.center) to (15.center);
		\draw [thick] (-1.center) to (15.center);
		\draw [thick] (9.center) to (16.center);
		\draw [thick] (7.center) to (16.center);
		\draw [thick] (10.center) to (17.center);
		\draw [thick] (17.center) to (2.center);
		\draw [thick] (17.center) to (3.center);
		\draw [thick] (3.center) to (18.center);
		\draw [thick] (18.center) to (12.center);
		\draw [thick] (-2.center) to (21.center);
		\draw [thick] (-1.center) to (21.center);
		\draw [thick] (21.center) to (19.center);
		\draw [thick] (1.center) to (22.center);
		\draw [thick] (2.center) to (22.center);
		\draw [thick] (22.center) to (11.center);
		\draw [thick] (22.center) to (13.center);
		\draw [thick] (23.center) to (14.center);
		\draw [thick] (28.center) to (0.center);
		\draw [thick] (0.center) to (24.center);
		\draw [thick] (29.center) to (1.center);
		\draw [rounded corners=5pt, dashed, color=gray] 
   (node cs:name=6, anchor=north west) --
   (node cs:name=8, anchor=south west) --
   (node cs:name=-1, anchor=south east) --
   (node cs:name=-2, anchor=north east) --
   cycle;
		\draw [rounded corners=5pt, dashed, color=gray] 
   (node cs:name=9, anchor=north west) --
   (node cs:name=7, anchor=south west) --
   (node cs:name=7, anchor=south east) --
   (node cs:name=9, anchor=north east) --
   cycle;
		\draw [rounded corners=5pt, dashed, color=gray] 
   (node cs:name=10, anchor=north west) --
   (node cs:name=10, anchor=south west) --
   (node cs:name=3, anchor=south east) --
   (node cs:name=2, anchor=north east) --
   cycle;
		\draw [rounded corners=5pt, dashed, color=gray] 
   (node cs:name=-2, anchor=north west) --
   (node cs:name=-1, anchor=south west) --
   (node cs:name=19, anchor=south east) --
   (node cs:name=19, anchor=north east) --
   cycle;
		\draw [rounded corners=5pt, dashed, color=gray] 
   (node cs:name=0, anchor=north west) --
   (node cs:name=0, anchor=south west) --
   (node cs:name=0, anchor=south east) --
   (node cs:name=0, anchor=north east) --
   cycle;
		\draw [rounded corners=5pt, dashed, color=gray] 
   (node cs:name=1, anchor=north west) --
   (node cs:name=1, anchor=south west) --
   (node cs:name=1, anchor=south east) --
   (node cs:name=1, anchor=north east) --
   cycle;
		\draw [rounded corners=5pt, dashed, color=gray] 
   (node cs:name=1, anchor=north west) --
   (node cs:name=2, anchor=south west) --
   (node cs:name=13, anchor=south east) --
   (node cs:name=11, anchor=north east) --
   cycle;
		\draw [rounded corners=5pt, dashed, color=gray] 
   (node cs:name=3, anchor=north west) --
   (node cs:name=3, anchor=south west) --
   (node cs:name=12, anchor=south east) --
   (node cs:name=12, anchor=north east) --
   cycle;
		\draw [rounded corners=5pt, dashed, color=gray] 
   (node cs:name=14, anchor=north west) --
   (node cs:name=14, anchor=south west) --
   (node cs:name=14, anchor=south east) --
   (node cs:name=14, anchor=north east) --
   cycle;
	\end{pgfonlayer}
\end{tikzpicture}
\end{aligned}
\:
  =
\:
\begin{aligned}
\begin{tikzpicture}[circuit ee IEC]
	\begin{pgfonlayer}{nodelayer}
		\node [style=none] (0) at (-2.75, -0) {$X$};
		\node [style=none] (1) at (1.75, -0) {$Z$};
		\node [contact, outer sep=5pt] (2) at (-2, 1) {};
		\node [contact, outer sep=5pt] (3) at (-2, -0.5) {};
		\node [contact, outer sep=5pt] (4) at (-2, 0.5) {};
		\node [contact, outer sep=5pt] (5) at (-2, -0) {};
		\node [contact, outer sep=5pt] (6) at (-2, -1) {};
		\node [contact, outer sep=5pt] (7) at (1, -0) {};
		\node [contact, outer sep=5pt] (8) at (1, -1) {};
		\node [contact, outer sep=5pt] (9) at (1, -0.5) {};
		\node [contact, outer sep=5pt] (10) at (1, 0.5) {};
		\node [style=none] (11) at (-0.5, 0.875) {};
		\node [style=none] (12) at (-0.5, 0.3) {};
		\node [contact, outer sep=5pt] (13) at (1, 1) {};
		\node [style=none] (14) at (-0.5, -0.2) {};
		\node [style=none] (15) at (-0.5, -0.6) {};
	\end{pgfonlayer}
	\begin{pgfonlayer}{edgelayer}
		\draw [thick] (2.center) to (11.center);
		\draw [thick] (4.center) to (11.center);
		\draw [thick] (11.center) to (13.center);
		\draw [thick] (5.center) to (14.center);
		\draw [thick] (3.center) to (14.center);
		\draw [thick] (15.center) to (7.center);
		\draw [thick] (15.center) to (9.center);
		\draw [thick] (6.center) to (15.center);
		\draw [thick] (15.center) to (8.center);
		\draw [thick] (12.center) to (10.center);
		\draw [rounded corners=5pt, dashed, color=gray] 
   (node cs:name=2, anchor=north west) --
   (node cs:name=4, anchor=south west) --
   (node cs:name=13, anchor=south east) --
   (node cs:name=13, anchor=north east) --
   cycle;
		\draw [rounded corners=5pt, dashed, color=gray] 
   (node cs:name=5, anchor=north west) --
   (node cs:name=3, anchor=south west) --
   (node cs:name=3, anchor=south east) --
   (node cs:name=5, anchor=north east) --
   cycle;
		\draw [rounded corners=5pt, dashed, color=gray] 
   (node cs:name=10, anchor=north west) --
   (node cs:name=10, anchor=south west) --
   (node cs:name=10, anchor=south east) --
   (node cs:name=10, anchor=north east) --
   cycle;
		\draw [rounded corners=5pt, dashed, color=gray] 
   (node cs:name=6, anchor=north west) --
   (node cs:name=6, anchor=south west) --
   (node cs:name=8, anchor=south east) --
   (node cs:name=7, anchor=north east) --
   cycle;
	\end{pgfonlayer}
\end{tikzpicture}
\end{aligned}
\]
Again, the composition law captures the idea that connectivity is all that
matters: as long as the wires are `ideal', the exact path does not matter. The
application to electrical circuits is discussed in detail in \cite{BF}.

This visualisation mimics the string diagrams defining extraspecial commutative
Frobenius monoids.

\section{Extraspecial commutative Frobenius monoids} \label{sec.escfm}

We introduce extraspecial commutative Frobenius monoids in some detail, writing
our axioms using the string calculus for monoidal categories introduced by Joyal
and Street \cite{JS}. Diagrams will be read left to right, and we shall suppress
the labels as we deal with a unique generating object and a unique generator of
each type. While we expect that the algebraic structures below---monoids,
comonoids, and so on---are familiar to most readers, we include the additional
detail to underscore the similarity between the wire diagrams for corelations
and string diagrams for extraspecial commutative Frobenius monoids. Again, we
shall see that the laws defining this structure express the principle that
connectivity is all that matters.

Recall that a \define{commutative monoid} $(X,\mu,\eta)$ in symmetric monoidal
category $(\mathcal C,\otimes)$ is an object $X$ of $\mathcal C$ together with
maps 
\[
  \xymatrixrowsep{1pt}
  \xymatrix{
    \mult{.075\textwidth} & & \unit{.075\textwidth} \\
    \mu\maps X\otimes X \to X & & \eta\maps I \to X
  }
\]
obeying
\[
  \xymatrixrowsep{1pt}
  \xymatrixcolsep{25pt}
  \xymatrix{
    \assocl{.1\textwidth} = \assocr{.1\textwidth} & \unitl{.1\textwidth} =
    \idone{.1\textwidth} & \commute{.1\textwidth} = \mult{.07\textwidth} \\
    \textrm{(associativity)} & \textrm{(unitality)} & \textrm{(commutativity)}
  }
\]
where $\swap{1em}$ is the braiding on $X \otimes X$. In addition to the
`upper' unitality law above, the mirror image `lower' unitality law also holds,
due to commutativity and the naturality of the braiding.  

Dually, a \define{cocommutative comonoid} $(X,\mu,\eta)$ in $\mathcal C$ is an
object $X$ together with maps 
\[
  \xymatrixrowsep{1pt}
  \xymatrix{
    \comult{.075\textwidth} & & \counit{.075\textwidth} \\
    \delta\maps X\to X \otimes X & & \epsilon\maps X \to I
  }
\]
obeying
\[
  \xymatrixrowsep{1pt}
  \xymatrixcolsep{25pt}
  \xymatrix{
    \coassocl{.1\textwidth} = \coassocr{.1\textwidth} & \counitl{.1\textwidth} =
    \idone{.1\textwidth} & \cocommute{.1\textwidth} = \comult{.07\textwidth} \\
    \textrm{(coassociativity)} & \textrm{(counitality)} &
    \textrm{(cocommutativity)}
  }
\]
Given a monoid and comonoid on the same object, there are two well-known ways
for them to interact: the bimonoid laws and the Frobenius law. We shall discuss
both in this paper, but for now we restrict our attention to Frobenius structure. 

\begin{definition}
  An \define{extraspecial commutative Frobenius monoid}
  $(X,\mu,\eta,\delta,\epsilon)$ in a monoidal category $(\mathcal C, \otimes)$
  comprises a commutative monoid $(X,\mu,\eta)$ and a cocommutative comonoid
  $(X,\delta,\epsilon)$ that further obey
  \[
  \xymatrixrowsep{1pt}
  \xymatrixcolsep{25pt}
  \xymatrix{
    \frobs{.1\textwidth} = \frobx{.1\textwidth} = \frobz{.1\textwidth} & \spec{.1\textwidth} =
    \idone{.1\textwidth} & \extral{.1\textwidth} = \extrar{.1\textwidth} \\
    \textrm{(the Frobenius law)} & \textrm{(the special law)} &
    \textrm{(the extra law)}
  }
  \]
\end{definition}
While we write two equations for the Frobenius law, this is redundant: the
equality of any two of the expressions implies the equality of all three.  Note
that a monoid and comonoid obeying the Frobenius law is commutative if and only
if it is cocommutative.  Thus while a commutative and cocommutative Frobenius
monoid might more properly be called a bicommutative Frobenius monoid, there is
no ambiguity if we only say commutative.

The Frobenius law and the special law go back to Carboni and Walters, under the
names S=X law and the diamond=1 law respectively \cite{CW}. The extra law is a
more recent discovery, appearing first under this name in the work of Baez and
Erbele \cite{BE}, as the `bone law' in \cite{BSZ,FRS}, and as the `irredundancy
law' in \cite{Za}.

Observe that each of these equations equate string diagrams that connect
precisely the same elements of the domain and codomain. To wit, the
associativity, coassociativity, and Frobenius laws show that the order in which
we build a connected component through pairwise clustering is irrelevant, the
special law shows that having multiple connections between points is irrelevant,
and the extra law shows that `extra' components not connected to the domain or
codomain are irrelevant. 

Our main theorem will show that these equations are exactly those required to
have the converse: two morphisms built from the generators of an extraspecial
commutative Frobenius monoid are equal and if and only if their diagrams impose
the same connectivity relations on the disjoint union of the domain and
codomain. This formalises an extension of the well-known spider theorem for
special commutative Frobenius monoids \cite{CK,CPP}. 

\section{Props for theories} \label{sec.props}

Introduced by Mac Lane \cite{ML} to generalise Lawvere's algebraic theories
approach to universal algebra \cite{Law}, the theory of props provides a
framework to discuss algebraic structures with multi-input multi-output
operations.

\begin{definition}
  A \define{prop} is a symmetric strict monoidal category having the natural
  numbers as objects and tensor product given by addition. A morphism of props
  is a symmetric strict identity-on-objects monoidal functor.
  
  If ${\mathcal T}$ is a prop and $\mathcal C$ is a symmetric monoidal category,
  we define an \define{algebra of} ${\mathcal T}$ \define{in} $\mathcal C$ to be a
  symmetric strict monoidal functor ${\mathcal T}\rightarrow \mathcal C$. A
  morphism of algebras of ${\mathcal T}$ in $\mathcal C$ is a monoidal natural
  transformation between them.
\end{definition}

Props allow us to study (one-sorted) symmetric monoidal theories, like those of
monoids, groups, and rings. Instances of these structures arise as algebras, or
models, of props.

\begin{definition}
  A \define{symmetric monoidal theory} $T=(\Sigma, E)$ comprises a
signature $\Sigma$ and a set of equations $E$. A \define{signature} is a set of
generators, where a \define{generator} is a formal symbol $\sigma\maps n \to m$.
From a signature $\Sigma$, we may formally construct the set of $\Sigma$-terms.
Defined inductively, a $\Sigma$\define{-term} takes one of the following forms:
\begin{itemize}
  \item the empty term $\varnothing\maps 0 \to 0$, the unit $\mathrm{id} \maps 1 \to 1$,
    the braiding $\swap{1em}\maps 2 \to 2$;
  \item the generators $\alpha\maps n \to m$ in $\Sigma$; 
  \item $\beta \circ \alpha\maps n \to p$, where $\alpha\maps n \to m$ and
    $\beta\maps m \to p$ are $\Sigma$-terms; or 
  \item $\alpha+\gamma\maps n+p \to m+q$, where $\alpha\maps n
    \to m$ and $\gamma\maps p \to q$ are $\Sigma$-terms.
\end{itemize}
We call $(n,m)$ the \define{type} of a $\Sigma$-term $\tau\maps n \to m$. An
\define{equation} is then a pair of two $\Sigma$-terms with the same type.

A \define{model} for a symmetric monoidal theory in a symmetric monoidal
category $(\mathcal C, \otimes)$ is an object $X$ together with morphisms
$\sigma_X \maps X^{\otimes n} \to X^{\otimes m}$ for every generator
$\sigma\maps n \to m$ in $\Sigma$, such that for every equation the two
$\Sigma$-terms are equal interpreted as morphisms in $\mathcal C$. A morphism
of models $X$ to $Y$ is a morphism $f\maps X \to Y$ in $\mathcal C$ such that
for every generator $\sigma$ we have $f^{\otimes m} \circ \sigma_X = \sigma_Y
\circ f^{\otimes n}\maps X^{\otimes n} \to Y^{\otimes m}$. 
\end{definition}

Many common algebraic structures can be expressed as symmetric monoidal
theories, including all those discussed in the previous section. For example,
the symmetric monoidal theory of commutative monoids has signature $\{\mu\maps 2
\to 1,\, \eta\maps 0 \to 1\}$ and three equations: precisely those pairs of terms
depicted in Section~\ref{sec.escfm}.

\begin{definition}
  We say that a prop ${\mathcal T}$ \define{is the prop for} a symmetric monoidal
  theory $T$ if for all symmetric monoidal categories $\mathcal C$ the category of
  algebras of ${\mathcal T}$ in $\mathcal C$ is equivalent to the category of
  models of $T$ in $\mathcal C$. 
\end{definition}

\begin{example}
  Write $\finset$ for the category of finite sets and functions, and also for
  its equivalent skeleton. This category inherits symmetric monoidal structure
  from the existence of finite coproducts, in this case the disjoint union of
  sets. Fixing a skeleton and, for example, utilising a total order on each set,
  one may choose the unitors and associator to be the identity, resulting a
  symmetric strict monoidal category \cite{Bu}. Thus we may consider $\finset$
  to be a prop.

  It is known that $\finset$ is the prop for commutative monoids \cite{Gr,Pi}.
  Indeed, write $m\maps 2 \to 1$ and $e\maps 0 \to 1$ for the unique maps of
  their type in ${\finset}$. Then given a symmetric monoidal functor
  $F:{\finset} \to \mathcal C$, the tuple $(F1,Fm,Fe)$ is a commutative monoid.
  Conversely, any commutative monoid $(X,\mu,\eta)$ in $\mathcal C$ gives rise
  to a functor ${\finset} \to \mathcal C$ mapping $1$ to $X$, $m$ to $\mu$, and
  $e$ to $\eta$.
\end{example}

We may bootstrap on this construction to show that $\corel$ has a strict
skeleton, and so too may be considered a prop. First, recall that in any
finitely cocomplete category $\mathcal C$ we may construct a symmetric monoidal
bicategory with the same objects and monoidal product, with morphisms cospans in
$\mathcal C$, composition of morphisms given by pushout, and with 2-morphisms
maps between apexes of cospans that commute over the feet \cite{Be}.
Decategorifying, we obtain a monoidal category $\mathbf{Cospan}(\mathcal{C})$,
where morphisms are isomorphism classes of cospans in $\mathcal C$.

Next, call a cospan $X \to N \leftarrow Y$ \define{jointly epic} if the induced
morphism $X+Y \to N$ is an epimorphism. If monomorphisms in $\mathcal C$ are
preserved under pushout, we may construct a symmetric monoidal category
$\mathbf{Corel}(\mathcal{C})$ with objects again those of $\mathcal C$, but this
time morphisms isomorphism classes of \emph{jointly epic} cospans in $\mathcal
C$, and composition taking the pushout of representative cospans, before
corestricting to the jointly epic part \cite{Mi,JW}.\footnote{
  More generally, a category of corelations may
  be constructed from any finitely cocomplete category equipped with a
  $(\mathcal E, \mathcal M)$-factorisation system such that $\mathcal M$ is
  preserved under pushout \cite{JW}. In related papers, we have shown that this
  construction can be used to model interconnection of `black-boxed' systems; that
  is, to model systems in which only the internal structure is obscured, leaving
  only the external behaviour \cite{BF,FRS,Fo}.}
The unitors, associator, and braiding are inherited from $\mathcal C$.

Our category $\corel$ can be constructed in this way.
\begin{theorem}
  The category $\corel$ is isomorphic as a symmetric monoidal category to
  $\corel(\finset)$.
\end{theorem}
\begin{proof}
  By the universal property of the coproduct, corelations $X+Y \to A$ are in
  one-to-one correspondence with jointly epic cospans $X \rightarrow A
  \leftarrow Y$. It is straightforward to check the notions of composition
  agree: consider the wire diagrams for corelations.
\end{proof}

As equivalences preserve colimits, replacing $\finset$ with its strict skeleton
thus shows that $\corel$ also has a strict skeleton. We henceforth use $\corel$
to refer to this equivalent prop. This allows us to restate our main theorem as
follows.

\begin{theorem} \label{thm.main}
  ${\mathbf{Corel}}$ is the prop for extraspecial commutative Frobenius monoids.
\end{theorem}

To prove this theorem, we begin by giving a more explicit construction of the
prop for extraspecial commutative Frobenius monoids.

\section{Props from theories} \label{sec.theories}

If we consider the set $\mathbb N \times \mathbb N$ as a discrete category, then
a signature is a functor from $\mathbb N \times \mathbb N$ to the category
$\Set$ of sets and functions. Note that to each prop we may associate the
so-called underlying signature $\hom(\cdot,\cdot)\maps \mathbb N \times \mathbb
N \to \Set$. The following important result allows us to understand the category
$\mathbf{PROP}$ of props; a proof can be found in Rebro \cite{Re} and Trimble
\cite{Tr}.

\begin{proposition}
  The underlying signature functor $U\maps \mathbf{PROP} \to \Set^{\mathbb N
  \times \mathbb N}$ is monadic. 
\end{proposition}

We write the right adjoint of this functor $F:\Set^{\mathbb N \times \mathbb N}
\rightarrow \mathbf{PROP}$, and call $F\Sigma$ the \define{free prop} on the
signature $\Sigma$. In fact, the free prop on $\Sigma$ has as morphisms the set
of $\Sigma$-terms \cite{Re}. 

Another important corollary of this theorem is that the category of props
is cocomplete. In particular, this allows us to take coequalizers in the
category of props. We use this to give an explicit construction of the prop for
a symmetric monoidal theory.

Let $(\Sigma,E)$ be a symmetric monoidal theory. Recall that each equation has a
type, and abuse notation to write $E$ also for the resulting signature. Then, as
the morphisms in $F\Sigma$ are $\Sigma$-terms and as $U$ and $F$ are adjoint, we
may define functors $\lambda,\rho\maps FE \to F\Sigma$ mapping each equation to
the first element and the second element of the pair respectively. This allows
us to build the prop for the theory.

\begin{proposition}
  The prop for a symmetric monoidal theory $(\Sigma,E)$ is the coequalizer of
  the diagram
  \[
    \xymatrix{
      FE \ar@<-.5ex>[r]_{\rho} \ar@<.5ex>[r]^{\lambda} & F\Sigma.
    }
  \]
\end{proposition}

Again, a proof may be found in Rebro \cite{Re} or Trimble \cite{Tr}. The
intuition is that the coequalizer is the weakest prop that forces the `left-hand
side' (given by $\lambda$) of each equation to equal the `right' (given by
$\rho$).

Write ${\mathbf{Th(ESCFM)}}$ for the prop for extraspecial commutative
Frobenius monoids constructed in this way. It remains to prove that this prop is
isomorphic to ${\corel}$.

\section{Corelations are the prop for extraspecial commutative Frobenius
monoids} \label{sec.mainthm}

In the influential paper \cite{La}, Lack develops the theory of distributive
laws for props, and proves the following as an example.  Note we write
${\cospan}$ for $\cospan({\finset})$.

\begin{proposition} \label{prop:cospanfinset2}
  ${\cospan}$ is isomorphic to the prop $\mathbf{Th(SCFM)}$
  for special commutative Frobenius monoids.
\end{proposition}

As the name suggests, a special commutative Frobenius monoid is a commutative
monoid and cocommutative comonid that further obey the Frobenius and special
laws. Note that in ${\finset}$ there are unique maps $0 \to 1$, $1 \to
1$, and $2 \to 1$. The isomorphism acts as follows on the generators
of $\mathbf{Th(SCFM)}$: 
\begin{align*}
  \alpha\maps \mathbf{Th(SCFM)} &\longrightarrow \cospan({\finset}); \\
  \mult{.05\textwidth} &\longmapsto \big(2 \to 1 \leftarrow 1\big) \\
  \unit{.05\textwidth} &\longmapsto \big(0 \to 1 \leftarrow 1\big) \\ 
  \comult{.05\textwidth} &\longmapsto \big(1 \to 1 \leftarrow 2\big) \\
  \counit{.05\textwidth} &\longmapsto \big(1 \to 1 \leftarrow 0\big).
\end{align*}
We use this to prove the main theorem.  The guiding intuition is that to
corestrict cospans to corelations is to impose the `extra' condition upon a
special commutative Frobenius monoid.

Our strategy will be to prove that $\mathbf{Th(ESCFM)}$ and $\corel$ are
coequalizers of isomorphic diagrams, and hence themselves isomorphic. First, we
show how to construct $\mathbf{Th(ESCFM)}$ as a coequalizer of props. 

\begin{lemma} \label{lem.coeqfrobmon} 
  The following is a coequalizer diagram:
\[
  \xymatrixrowsep{5pt}
  \xymatrixcolsep{35pt}
  \xymatrix{
    FE_{\mathrm{Ex}} \ar@<0.6ex>[r]^(.4)\lambda \ar@<-0.6ex>[r]_(.4)\rho
    &  \mathbf{Th(SCFM)} \ar[r] & \mathbf{Th(ESCFM)}.
  }
\]
\end{lemma}
\begin{proof}
  Let $(\Sigma,E_{\mathrm{SCFM}})$ and $(\Sigma,E_{\mathrm{ESCFM}})$ be the
  theories of special commutative Frobenius monoids and extraspecial commutative
  Frobenius monoids respectively---note that they have the same set of
  generators, $\Sigma$. Write also $(\Sigma,E_{\mathrm{Ex}})$ for the theory of
  the `extra law', so $E_{\mathrm{Ex}}$ contains just a single element
  $\bullet\maps 0 \to 0$. This has image $\lambda_{\mathrm{Ex}}(\bullet) =
  \begin{aligned}\extral{.09\textwidth}\end{aligned}$ and
  $\rho_{\mathrm{Ex}}(\bullet) = \varnothing$
  under the two canonical maps $FE_{\mathrm{Ex}} \rightrightarrows F\Sigma$. 
  
  Now, by construction we have a map $F\Sigma \to \mathbf{Th(SCFM)}$, and
  composing this with $\lambda_{\mathrm{Ex}}$ and $\rho_{\mathrm{Ex}}$ gives
  $\lambda$ and $\rho$ respectively. Since an extraspecial commutative Frobenius
  monoid is a fortiori a special commutative Frobenius monoid, by construction
  we also have a map $\mathbf{Th(SCFM)} \to \mathbf{Th(ESCFM)}$; this is the
  unlabelled map above.  As $E_{\mathrm{ESCFM}} = E_{\mathrm{SCFM}}+
  E_{\mathrm{Ex}}$, it is straightforward to verify that the above diagram is a
  coequalizer diagram.
\end{proof}
Next, we construct $\corel$ as a coequalizer.
\begin{lemma} \label{lem.coeqcospan}
  The following is a coequalizer diagram:
\[
  \xymatrixrowsep{5pt}
  \xymatrixcolsep{35pt}
  \xymatrix{
    FE_{\mathrm{Ex}} \ar@<0.6ex>[r]^(.45){\alpha \circ \lambda}
    \ar@<-0.6ex>[r]_(.45){\alpha \circ \rho}
    & {\cospan} \ar[r] & {\corel}
  }
\]
\end{lemma}
\begin{proof}
  The map ${\cospan} \to {\corel}$ is the canonical one corestricting each
  cospan to its jointly epic part. It is straightforward to check this is indeed
  a map of props; details can be found in \cite{Fo}. Now $(\alpha \circ
  \lambda)(\bullet) = (0 \to 1 \leftarrow 0)$, while $(\alpha
  \circ\rho)(\bullet) = (0 \to 0 \leftarrow 0)$. This implies the above diagram
  commutes from $FE_{\mathrm{Ex}}$ to ${\corel}$. It remains to check the
  universal property.

  Suppose that we have a prop $\mathcal T$ such that the diagram
  \[
    \xymatrixrowsep{5pt}
    \xymatrixcolsep{35pt}
    \xymatrix{
      FE_{\mathrm{Ex}} \ar@<0.6ex>[r]^(.45){\alpha \circ \lambda}
      \ar@<-0.6ex>[r]_(.45){\alpha \circ \rho}
      & {\cospan} \ar[r]^(.6)A & \mathcal T
    }
  \]
  commutes from $FE_{\mathrm{Ex}}$ to $\mathcal T$. We must show there is a
  unique map $A'\maps {\corel} \to \mathcal T$.

  As the map ${\cospan} \to {\corel}$ is full, it is enough
  to show that each cospan $(n \stackrel{f}{\rightarrow} a
  \stackrel{g}{\leftarrow} m)$ has the same image as its jointly epic part $(n
  \stackrel{f'}{\rightarrow}\im[f,g] \stackrel{g'}{\leftarrow} m)$ under $A$---we then have a
  unique and well-defined map $A'$ sending each corelation to its image as a
  cospan under $A$. But this is straightforward:
  \begin{align*}
    A\big(n \stackrel{f}{\rightarrow} a \stackrel{g}{\leftarrow} m\big)
    &= A\big(n \stackrel{f'}{\rightarrow} \im[f,g] \stackrel{g'}{\leftarrow} m\big) +
    A\big(0 \rightarrow (a-\im[f,g]) \leftarrow 0\big)\\
    &= A\big(n \stackrel{f'}{\rightarrow} \im[f,g] \stackrel{g'}{\leftarrow} m\big) +
    (A\circ\alpha\circ\lambda)\big(\bullet^{+(a-\im[f,g])}\big) \\
    &= A\big(n \stackrel{f'}{\rightarrow} \im[f,g] \stackrel{g'}{\leftarrow} m\big) +
    (A\circ\alpha\circ\rho)\big(\bullet^{+(a-\im[f,g])}\big) \\
    &= A\big(n \stackrel{f'}{\rightarrow} \im[f,g] \stackrel{g'}{\leftarrow} m\big).
  \end{align*}
  This proves the lemma.
\end{proof}

\begin{proof}[of Theorem \ref{thm.main}.]
More explicitly now, our strategy is to show both ${\corel}$ and
$\mathbf{Th(ESCFM)}$ are coequalizers in the diagram
\[
  \xymatrixrowsep{5pt}
  \xymatrixcolsep{35pt}
  \xymatrix{
    &  \mathbf{Th(SCFM)}
    \ar[dd]^{\alpha} \ar[r] & \mathbf{Th(ESCFM)} \ar[dd] \\
    FE_{\mathrm{Ex}} \ar@<0.6ex>[ur]^\lambda \ar@<-0.6ex>[ur]_\rho
    \ar@<0.6ex>[dr]^{\alpha \circ \lambda}
    \ar@<-0.6ex>[dr]_{\alpha\circ\rho}\\
    & {\cospan} \ar[r] & {\corel}
  }
\]
Lemma \ref{lem.coeqfrobmon} shows the upper row is a coequalizer diagram, while
Lemma \ref{lem.coeqcospan} shows the lower row is too. As the two relevant
triangles commute and the first vertical map is an isomorphism,
$\mathbf{Th(ESCFM)}$ and ${\corel}$ are coequalizers of isomorphic
diagrams, and hence themselves isomorphic.
\end{proof}

The so-called spider theorem is an immediate corollory.
\begin{corollary}
  Two morphisms in an extraspecial commutative Frobenius monoid are equal if and
  only if they map to the same corelation.
\end{corollary}

\section{Spans, cospans, relations, corelations} \label{sec.summary}

Lastly, we return to the big picture. The dual theorems are known for spans and
relations \cite{La,WW}, but the above method of proof provides a novel argument,
and illuminates the duality. Recall bimonoids, sometimes also called bialgebras.

\begin{definition}
  A \define{bicommutative bimonoid} $(X,\mu,\eta,\delta,\epsilon)$ in a
  monoidal category $(\mathcal C, \otimes)$ comprises a commutative monoid
  $(X,\mu,\eta)$ and a cocommutative comonoid $(X,\delta,\epsilon)$ that further
  obey the extra law and the bimonoid laws
  \[
  \xymatrixrowsep{1pt}
  \xymatrixcolsep{25pt}
  \xymatrix{
    \bi{.12\textwidth} = \frobx{.10\textwidth} & \bimultl{.08\textwidth} =
    \bimultr{.08\textwidth} & \bicomultl{.08\textwidth} = \bicomultr{.08\textwidth} \\
    & \textrm{(the bimonoid laws)} &
  }
  \]
\end{definition}

Bimonoids can be understood as dual to Frobenius monoids:
$\Span({\finset})$ is the prop for bicommutative bimonoids. This fact
goes back to Lack \cite{La}. Wadsley and Woods provide an alternative proof, via
the fact that the category of matrices over a rig\footnote{Also known as a
  semiring, a rig is a ring without the condition that additive inverses exist.
  That is, it is a ri\textbf{n}g without \textbf{n}egatives.} $R$ is the prop for 
bicommutative bimonoids equipped with an action of the rig $R$ \cite[Theorem
5]{WW}. Choosing the rig of booleans, this also implies that ${\rel}$,
the prop equivalent to the category of finite sets and relations, is
the prop for special bicommutative bimonoids. The techniques of this paper can
be co-opted to provide an alternate proof of this fact.

\begin{theorem}
  ${\rel}$ is isomorphic to the prop for special bicommutative
  bimonoids.
\end{theorem}
\begin{proof}
  To sketch: Lack has already shown, using the distributive law arising from
  pullbacks in ${\finset}$, that $\Span({\finset})$ is
  isomorphic to the prop for bicommutative bimonoids. We may use this to set up
  isomorphic coequalizer diagrams in the category $\mathrm{PROP}$ to obtain both
  the prop for special bicommutative bimonoids and the prop ${\rel}$.
  The isomorphism arises from the observation that taking the jointly monic part
  of a span is equivalent to iteratively asserting that the span $(1 \leftarrow
  2 \rightarrow 1)$ may be replaced by the identity $(1 \leftarrow 1 \rightarrow
  1)$, and that this manifests as the special law.
\end{proof}

We conclude by displaying our table once again, bringing out the symmetry
by annotating names with the suppressed aspects of their structure.
\[
\begin{tabular}{c|c}
  \textbf{spans} & \textbf{cospans} \\
  extra bicommutative & special bicommutative \\
  bimonoids & Frobenius monoids \\
  \hline
  \textbf{relations} & \textbf{corelations} \\
  extraspecial bicommutative & extraspecial bicommutative  \\
  bimonoids & Frobenius monoids
\end{tabular}
\]

\end{document}